\documentclass{amsart}
\usepackage{amsmath, amssymb,epic,graphicx,mathrsfs,enumerate}
\usepackage[all]{xy}

\usepackage{amsthm}
\usepackage{amssymb}
\usepackage{latexsym}
\usepackage{longtable}
\usepackage{epsfig}
\usepackage{amsmath}
\usepackage{hhline}

 \DeclareMathOperator{\perm}{Sym}
 \DeclareMathOperator{\soc}{soc}

 \DeclareMathOperator{\frat}{Frat}

\DeclareMathOperator{\core}{Core}
\DeclareMathOperator{\Ee}{E}

\DeclareMathOperator{\End}{End} 
\DeclareMathOperator{\der}{Der}

\newtheorem{thm}{Theorem}%[section]
\newtheorem{cor}[thm]{Corollary}
 \newtheorem{lemma}[thm]{Lemma}

\numberwithin{equation}{section}

\renewcommand{\footnote}{\endnote}
\newcommand{\ignore}[1]{}\makeglossary

\begin{document}
	\bibliographystyle{amsplain}
	\title[The expected number of elements to generate a group]{A bound on the expected number\\ of random elements to generate a finite  group
		\\all of whose Sylow subgroups are $d$-generated.}
	\author{Andrea Lucchini}
	\address{
		Andrea Lucchini\\ Universit\`a degli Studi di Padova\\  Dipartimento di Matematica\\ Via Trieste 63, 35121 Padova, Italy\\email: lucchini@math.unipd.it}
	\thanks{Partially supported by Universit\`a di Padova (Progetto di Ricerca di Ateneo: \lq\lq Invariable generation of groups\rq\rq).}

	\begin{abstract}Assume that all the Sylow subgroups of a finite group $G$ can be generated by $d$ elements. Then the expected number of elements of  $G$ which have to be drawn at random, with replacement,
		before a set of generators is found, is at most $d+\eta$ with
		$\eta \sim 2.875065.$
	\end{abstract}
	\maketitle
	\section{introduction}
	In 1989, R. Guralnick \cite{rg} and the author \cite{al} independently proved
	that if all the Sylow subgroups of a finite group $G$ can be generated by $d$ elements, then the group $G$ itself can be generated by $d+1$ elements.
	The aim of this paper is to obtain a probabilistic version of this result.
	
	\

	Let $G$ be a nontrivial finite group and let $x=(x_n)_{n\in\mathbb N}$ be a sequence of independent, uniformly distributed $G$-valued random variables.
	We may define a random variable $\tau_G$ by
	$\tau_G=\min \{n \geq 1 \mid \langle x_1,\dots,x_n \rangle = G\}.$
We denote by $e(G)$ the expectation $\Ee(\tau_G)$ of this random variable.
	In other word $e(G)$ is the expected number of elements of $G$ which have to be drawn at random, with replacement,
	before a set of generators is found.
	 Some estimations of the value $e(G)$ have been obtained in \cite{mon}. The main result of this paper is:

 \begin{thm}\label{main}Let $G$ be a finite group. If all the Sylow subgroups of $G$ can be generated by $d$ elements, then $$e(G)\leq d+\eta \quad \text { with } \quad \eta=\frac{5}{2}+\sum_{p\geq 3}\frac{1}{(p-1)^2}< 3.$$
 \end{thm}
 
From an accurate estimation of $\sum_p (p-1)^{-2}$ given in \cite{hc}, it follows 
$\eta \sim 2.875065...$
 	This result is near to be best possible. For any prime $p,$ let $A_{p,d}$ be the elementary abelian $p$-group of rank $d$ and for any positive integer $n$
 consider $A_{n,d}=\prod_{p\leq n}A_{p,d}.$ C. Pomerance \cite{pom} proved that $\lim_{n\to \infty} e(A_{n,d})=d+\sigma$, where $\sigma \sim 2.11846...$ (the exact value of $\sigma$ can be
 explicitly described in terms of the Riemann zeta--function).

\
 
  If $G$ is a $p$-subgroup of $\perm(n),$ then $G$ can be generated by $\lfloor n/p\rfloor$ elements (see \cite{kp}), so Theorem \ref{main} has the following consequence:
  	
\begin{cor}
	If $G$ is a permutation group of degree $n,$ then $e(G)\leq  \lfloor n/2\rfloor+\eta.$
\end{cor} 

A profinite group $G$, being a compact
topological group,  can be seen as a probability space.
If we denote with $\mu$ the
normalized Haar measure on $G$, so that $\mu(G)=1$, the
probability that $k$ random elements generate (topologically) $G$ is defined as
$$P_G(k)=\mu(\{ (x_1,\ldots,x_k)\in G^k|\langle x_1,\ldots,x_k\rangle
=G\}),$$ where $\mu$ denotes also the product measure on $G^k$.
The definition of $e(G)$ can be extended to finitely generated profinite groups. In particular $e(G)= \sup_{N\in \mathcal N}e(G/N)$, being $\mathcal N$ the set of the open normal subgroups of $G$ (see for example \cite[Section 6]{mon}), hence
Theorem \ref{main} remains true for profinite groups:\begin{sl} if all the Sylow subgroups of a profinite group $G$ are (topologically) $d$-generated, then $G$ is (topologically) $(d+1)$-generated and
$e(G)\leq d+\eta.$\end{sl}
A profinite group $G$ is said to be positively finitely generated, PFG for
short, if $P_G(k)$ is positive for some natural number $k$, and the
least such natural number is denoted by $d_P(G)$.
Not all finitely generated profinite groups are PFG (for example if $\hat F_d$ is the free profinite group of rank $d\geq 2$ then
$P_{\hat F_d}(t)=0$ for every $t\geq d$, see \cite{KL}): if $G$ is not PFG we set $d_P(G)=\infty.$ It can be easily seen that
$e(G)=\sum_{n\geq 0}1-P_G(n)$ (see (\ref{inizi1}) in Section \ref{main}). Since $P_G(n)=0$ whenever $n\leq d_P(G),$ we immediately deduce
that $e(G)> d_P(G).$ In particular, if all the Sylow subgroups of $G$ are $d$-generated, then $d_P(G)<e(G)<d+3,$ and therefore we obtain the following result:

\begin{thm}\label{profi}
If all the Sylow subgroups of a profinite group $G$ are (topologically) $d$-generated, then $d_P(G)\leq d+2.$
\end{thm}

The previous result is best possible. For a given $d\in \mathbb N,$ let $A_{p,d}$ be an elementary abelian $p$-group of rank $d$, $A_d=\prod_{p\neq 2}A_{p,d}$ and
consider the semidirect product $G_d=A_d\rtimes B,$ where $B=\langle b \rangle$ is cyclic of order 2 and $a^b=a^{-1}$ for every $a\in A_d.$ Clearly all the Sylow subgroups of $G$ are  $d$-generated. We claim that $d_P(G)=d+2.$ It follows from the main theorem in
\cite{g2} that for every $k \in \mathbb N,$ we have
$$P_{G_d}(k)=\left(1-\frac{1}{2^k}\right)\prod_{p\neq 2}\left(1-\frac{p}{p^k}\right)\cdots \left(1-\frac{p^d}{p^k}\right).$$
In particular $P_{G_d}(d+1)=0$ since the series
$$\sum_{p\neq 2}\left(\frac{p}{p^{d+1}}+\cdots+\frac{p^d}{p^{d+1}}\right)=\sum_{p\neq 2}\frac{p^{d+1}-1}{(p-1)p^{d+1}}$$ is divergent. But then $d_P(G)>d+1,$ hence, by Theorem \ref{profi}, we conclude $d_P(G)=d+2.$

\section{Proof of the main result}\label{main}

Let $G$ be a finite group and use the following notations:
\begin{itemize}
	\item For a given prime $p,$ $d_p(G)$ is the smallest cardinality of a generating set of a Sylow $p$-subgroup of $G.$
	%\item $d=\max_p d_p(G).$
	\item For a given prime $p$ and a positive  integer $t,$ $\alpha_{p,t}(G)$ is the number of complemented  factors of  order $p^t$ in a chief series of $G.$
	\item For a given prime $p,$ $\alpha_p(G)=\sum_t \alpha_{p,t}(G)$ is the number of complemented factors of $p$-power order in a chief series of $G.$
	\item $\beta(G)$ is the number of nonabelian factors  in a chief series of $G.$
\end{itemize}

\begin{lemma}\label{stime}For every finite group $G,$ we have:
\begin{enumerate} 
	\item $\alpha_p(G)\leq d_p(G).$
	\item $\alpha_2(G)+\beta(G)\leq d_2(G).$
	\item If $\beta(G)\neq 0$, then $\beta(G)\leq d_2(G)-1.$
\end{enumerate}
\end{lemma}
\begin{proof}
We prove the three statements by induction on the order of $G.$ Let $N$ 
be a minimal normal subgroup of $G.$
\begin{enumerate}
	\item If $N$ is not complemented in $G$ or $N$ is not a $p$-group, then, by induction,
	$$\alpha_p(G)=\alpha_p(G/N) \leq d_p(G/N)=d_p(G).$$
	Otherwise $G=N\rtimes H$ for a suitable $H\leq G$ and $\alpha_p(G)=\alpha_p(G/N)+1=\alpha_p(H)+1\leq d_p(H)+1\leq d_p(G).$
	\item  If $N$ is abelian, we argue as in (1). Assume that $N$ is nonabelian  and let $P$ be a Sylow
	2-subgroup of $G$. By Tate's Theorem \cite[p. 431]{hup}, $P\cap N\not\leq \frat P,$ and consequently $\beta(G)=\beta(G/N)+1\leq d_2(G/N)+1\leq d_2(G).$
		\item Suppose $\beta(G) \neq 0$. As before we may assume that $N$ is nonabelian and this implies $d_2(G/N)+1\leq d_2(G).$ If $\beta(G/N)\neq 0,$ then we easily conclude by induction. If $\beta(G/N)=0$ then $\beta(G)=1$
		while $d_2(G)\geq 2,$	
		since a Sylow 2-subgroup of a finite nonabelian simple group, and consequently of $N,$ is never cyclic.
		\qedhere 
	\end{enumerate}
\end{proof}

	Notice that $\tau_G>n$ if and only if $\langle x_1,\dots,x_n \rangle \neq G$, so we have $P(\tau_G>n)=1-P_G(n),$ denoting by $P_G(n)$ the probability that $n$ randomly chosen
	elements of $G$ generate $G.$
	Clearly we have:
	\begin{equation}\label{inizi1}
	\begin{aligned}e(G)&=\sum_{n\geq 1}nP(\tau_G=n)=\sum_{n\geq 1}\left(\sum_{m\geq n}P(\tau_G=m)\right)\\
	&=\sum_{n\geq 1}P(\tau_G\geq n)=\sum_{n\geq 0}P(\tau_G>n)=\sum_{n\geq 0}(1-P_G(n)).
	\end{aligned}
	\end{equation}
	
	Denote by $m_n(G)$ the number of index $n$ maximal subgroups of $G.$	We have (see \cite[11.6]{sub}):
	\begin{equation}\label{inizi2}1-P_G(k)\leq \sum_{n\geq 2}\frac{m_n(G)}{n^{k}}.
	\end{equation}

 Using the notations introduced in \cite[Section 2]{pak}, we say that a maximal subgroup $M$ of $G$ is of type A if $\soc(G/\core_G(M))$ is abelian, of type B otherwise, and we denote by $m^A_n(G)$ (respectively $m^B_n(G)$)  the number of maximal subgroups of $G$ of type A (respectively B) of
index $n.$  Given an irreducible $G$-group $V$, let $\delta_V(G)$ be the number of
complemented factors $G$-isomorphic to $V$ in a chief series of $G$ and $q_{V}(G)=|\End_{G}(V)|$.
Moreover, for $n\in \mathbb N,$ let $\mathcal A_n$ be the set of the irreducible $G$-modules
$V$ with $\delta_V(G)\neq 0$ and $|V|=n.$

\begin{lemma}\label{lab1}Let $n=p^t$ for some prime $p$. If $m_n^A(G)\neq 0,$ then $\alpha_{p,t}(G)\neq 0$ and $$m_n^A(G)\leq \frac{n^{\alpha_{p,t}(G)+1}}{p-1}.$$
\end{lemma}	
\begin{proof} For a given $V\in \mathcal{A}_n$, let $m_V(G)$ be the number of maximal subgroups $M$ of $G$ with $\soc(G/\core_G(M))\cong_G V.$
		From \cite[Section 2]{pak} and \cite[Section 4]{crpr} it follows that
$$m_V(G)\leq \frac{q_V(G)^{\delta_V(G)}-1}{q_V(G)-1}|\der(G/C_G(V),V)|.$$
By \cite[Theorem 1]{guho}, we have $|\der(G/C_G(V),V)|\leq |V|^{3/2}$. Moreover 
(see for example \cite[Lemma 1]{st})
$|\der(G/C_G(V),V)|\leq |V|$ if $G/C_G(V)$ is soluble, which happens in particular when $q_V(G)=n$ (indeed in this case $G/C_G(V)$ is isomorphic to a subgroup of the multiplicative group of the field of order $q_V(G)$). If
$q_V(G)\neq n,$ then $\dim_{\End_G(V)}V\geq 2$, hence $n=|V|\geq q_G(V)^2$ and 
consequently
$$m_V(G)\leq \frac{q_V(G)^{\delta_G(V)}n^{3/2}}{q_V(G)-1}\leq  \frac{n^{\delta_G(V)/2}n^{3/2}}{p-1}
\leq \frac{n^{\delta_G(V)+1}}{p-1}.$$ On the other hand, if $q_V(G)=n,$ then
$$m_V(G)\leq \frac{q_V(G)^{\delta_G(V)}n}{q_V(G)-1}\leq  \frac{n^{\delta_G(V)}n}{p-1}
\leq  \frac{n^{\delta_G(V)+1}}{p-1}.$$
We conclude
$$\begin{aligned}m_n^A(G)&=\sum_{V\in \mathcal{A}_n}m_V(G)\leq 
\frac{n}{p-1}\sum_{V\in \mathcal{A}_n}n^{\delta_V(G)}
\leq 
\frac{n}{p-1}\prod_{V\in \mathcal{A}_n}n^{\delta_V(G)}\\&=\frac{n^{1+\sum_{V\in \mathcal{A}_n}\delta_V(G)}}{p-1}= \frac{n^{\alpha_{p,t}(G)+1}}{p-1}.\qedhere\end{aligned}$$
\end{proof}

\begin{lemma}\label{lnab1}
If $m_n^B(G)\neq 0,$ then $n\geq 5$, $\beta(G)\neq 0$ and $m_n^B(G)\leq \frac{\beta(G)(\beta(G)+1)n^2}{2}.$
\end{lemma}

\begin{proof}
The condition $n\geq 5$ follows from the fact there there is no unsoluble primitive permutation group of degree $n<5.$ The remaining part of the statement follows from
\cite[Claim 2.4]{pak}.
\end{proof}

 \begin{lemma} Let $d=\max_p d_p(G)$ 
 	and let $$\mu_p(G)=\sum_{k\geq d+2}\left(\sum_{t\geq 1}\frac{m_{p^t}^A(G)}{p^{tk}}\right).$$
 	If $\alpha_p(G)=0,$ then $\mu_p(G)=0.$ Otherwise
$$\mu_p(G)
\leq \begin{cases}\frac{1}{p^{d-\alpha_p(G)}}\frac{1}{(p-1)^2}\leq \frac{1}{(p-1)^2}&\text{if $p$ is odd,}\\ \frac{1}{2^{d-\alpha_2(G)}}\frac{1}{2}\leq \frac{1}{2}&\text{otherwise.}\end{cases}$$
\end{lemma}
 \begin{proof}First notice that, by Lemma \ref{stime}, we have 
 	$\alpha_p(G)\leq d_p(G)\leq d.$
 	Let $\theta_{p,t}=0$ if $\alpha_{p,t}(G)=0,$ $\theta_{p,t}=1$ otherwise. 
 	 By Lemma \ref{lab1} we have
$$\begin{aligned}\sum_{k\geq d+2}&\left(\sum_{t\geq 1}\frac{m_{p^t}^A(G)}{p^{tk}}\right)\leq \frac{p}{p-1}\sum_{k\geq d+2}\left(\sum_{t\geq 1}\frac{p^{t\alpha_{p,t}(G)}\theta_{p,t}}{p^{tk}}\right)\\&\leq \sum_{k\geq d+2}\frac{p}{p-1}\left(\sum_{t\geq 1}\frac{p^{\alpha_{p,t}(G)}\theta_{p,t}}{p^{k}}\right)
\leq \frac{p}{p-1}\sum_{k\geq d+2}\left(\frac{p^{\sum_{t\geq 1}\alpha_{p,t}(G)}}{p^k}\right)\\&\leq \frac{p}{p-1}\sum_{k\geq d+2}\frac{p^{\alpha_p(G)}}{p^k}
\leq \frac{p}{p-1}\sum_{k\geq d+2}\frac{p^{d}}{p^kp^{d-\alpha_p(G)}}\\
&\leq \frac{p}{p^{d-\alpha_p(G)}(p-1)}\sum_{u\geq 2}\frac{1}{p^u}\leq \frac{1}{p^{d-\alpha_p(G)}}\frac{1}{(p-1)^2}.
\end{aligned}$$
Notice that, for $k>d\geq d_2(G)\geq \alpha_{2,t}(G),$ we have
$$\frac{m_{2^t}^A(G)}{2^{tk}}\leq \frac{2^{t\alpha_{2,t}(G)+1}}{2^{tk}}\leq \frac{2^{\alpha_{2,t}(G)}}{2^k} \text \quad { if } \quad t>1.
$$
On the other hand,
$$
m_2^A(G)=2^{\alpha_{2,1}(G)}-1\leq 2^{\alpha_{2,1}(G)}.
$$
Hence
$$\begin{aligned}\sum_{k\geq d+2}&\left(\sum_{t\geq 1}\frac{m_{2^t}^A(G)}{2^{tk}}\right)\leq \sum_{k\geq d+2}\left(\sum_{t\geq 1}\frac{2^{\alpha_{2,t}(G)}\theta_{2,t}}{2^{k}}\right)
\leq \sum_{k\geq d+2}\left(\frac{2^{\sum_{t\geq 1}\alpha_{2,t}(G)}}{2^k}\right)\\&\leq \sum_{k\geq d+2}\frac{2^{\alpha_2(G)}}{2^k}
\leq \sum_{k\geq d+2}\frac{2^{d}}{2^k2^{d-\alpha_2(G)}}
\leq \frac{1}{2^{d-\alpha_2(G)}}\sum_{u\geq 2}\frac{1}{2^u}\leq \frac{1}{2^{d-\alpha_p(G)}}\frac{1}{2}.\qedhere
\end{aligned}$$
\end{proof}

\begin{lemma}Let $d=\max_p d_p(G)$  and let 
	$$\mu^*(G)=\sum_{k\geq d+2}\left(\sum_{n\geq 5} \frac{m_n^B(G)}{n^k}\right).$$
If $\beta(G)=0,$ then $\mu^*(G)=0.$ Otherwise	
$$\mu^*(G)\leq 
	 \frac{1}{4\cdot5^{d-(\beta(G)+1)}}\leq \frac{1}{4}.$$
\end{lemma}
\begin{proof} Notice that, by Lemma \ref{stime},  if $\beta(G)\neq 0,$ then $d\geq d_2(G)\geq \beta(G)+1 \geq 2.$ We deduce
	from Lemma \ref{lnab1} that
 $$\begin{aligned}
\sum_{k\geq d+2}&\left(\sum_{n\geq 5} \frac{m_n^B(G)}{n^k}\right)
 \leq \sum_{k\geq d+2}\left(\sum_{n\geq 5} \frac{\beta(G)(\beta(G)+1)n^2}{2n^k}\right)\\
 &\leq \frac{\beta(G)(\beta(G)+1)}2\sum_{u\geq 2}\left(\sum_{n\geq 5}\frac{n^2}{n^{d+u}}\right)
\leq  \frac{\beta(G)(\beta(G)+1)}{2\cdot 5^{d-2}}\sum_{u\geq 2}\left(\sum_{n\geq 5}\frac{1}{n^u}\right)
 \\ &\leq  \frac{\beta(G)(\beta(G)+1)}{2\cdot 5^{d-2}}\sum_{n\geq 5}\left(\sum_{u\geq 2}\frac{1}{n^u}\right)\leq \frac{\beta(G)(\beta(G)+1)}{2\cdot 5^{d-2}} \sum_{n\geq 5}\frac{1}{n^2}\frac{n}{n-1}
 \\&=\frac{\beta(G)(\beta(G)+1)}{2\cdot 5^{d-2}}\sum_{n\geq 4}\frac{1}{n(n+1)}
=\frac{\beta(G)(\beta(G)+1)}{2\cdot 5^{\beta(G)-1}\cdot 5^{d-(\beta(G)+1)}}\cdot \frac{1}{4}\\&\leq \frac{1}{4\cdot 5^{d-(\beta(G)+1)}}.\qedhere \end{aligned}$$
 \end{proof}

\begin{lemma}We have $\mu_2(G)+\mu^*(G)\leq 1/2$.
%$$\sum_{k\geq d+2}\left(\sum_{t\geq 1}\frac{m_{2^t}^A(G)}{2^{tk}}\right)+	\sum_{k\geq d+2}\left(\sum_{n\geq 5} \frac{m_n^B(G)}{n^k}\right)
%	\leq \frac{1}{2}.$$
\end{lemma}
\begin{proof}
By Lemma \ref{stime}, $\alpha_2(G)+\beta(G)\leq d_2(G)\leq d.$ If $d=\alpha_2(G)$ then $\beta(G)=0$, and consequently 
$$\mu_2(G)+\mu^*(G)=\mu_2(G)\leq  \frac{1}{2}.$$
In the remain cases, we have
$$\mu_2(G)+\mu^*(G)\leq
\frac{1}{2\cdot 2^{d-\alpha_p(G)}}+\frac{1}{4\cdot 5^{d-(\beta(G)+1)}}\leq \frac{1}{4}+\frac{1}{4}=\frac 1 2.\qedhere$$
\end{proof}

\begin{proof}[Proof of Theorem \ref{main}] From (\ref{inizi1}), (\ref{inizi2}) and the last three lemmas, we deduce
$$\begin{aligned}e(G)&=\sum_{k\geq 0}(1-P_G(k))\leq  d+2+ \sum_{k\geq d+2}(1-P_G(k))\\
&\leq d+2+\sum_p\left(\sum_{k\geq d+2}\left(\sum_{t\geq 1}\frac{m_{p^t}^A(G)}{p^{tk}}\right)\right)
+\sum_{k\geq d+2}\left(\sum_{n\geq 5} \frac{m_n^B(G)}{n^k}\right)\\
&=d+2+\sum_p \mu_p(G)+\mu^*(G) \leq d+\frac{5}{2}+ \sum_{p>2} \frac{1}{(p-1)^2}.\qedhere
\end{aligned}$$
\end{proof}


\begin{thebibliography}{99}
\bibitem{hc} H. Cohen, High precision computation of Hardy-Littlewood constants, preprint available on the author's web page. 
\bibitem{crpr}  E. Detomi and A. Lucchini, Crowns in profinite groups and applications, Noncommutative algebra and geometry, 47–-62, Lect. Notes Pure Appl. Math., 243, Chapman  Hall/CRC, Boca Raton, FL, 2006. 
\bibitem{hup} B.	Huppert, 
	Endliche Gruppen,
	Die Grundlehren der Mathematischen Wissenschaften, Band 134 Springer-Verlag, Berlin-New York 1967 
	\bibitem{g2}
	{W. Gasch\"utz}, Die Eulersche Funktion endlicher aufl\"{o}sbarer Gruppen,  Illinois J. Math.  3  (1959),
	469--476.
\bibitem{rg} R. Guralnick,
On the number of generators of a finite group,
Arch. Math. 53 (1989), no. 6, 521-523.
\bibitem{guho} R. Guralnick and C. Hoffman, The first cohomology group and generation of simple groups, in: Proc. Conf. Groups and Geometries, Sienna, 1996, in: Trends Math., Birkhäuser, Basel, 1998, pp. 81-–89.
\bibitem{KL}
W. M. Kantor and A. Lubotzky, The probability of generating a finite
classical group, {Geom. Ded.}
{36},  (1990) 67--87.
\bibitem{kp} L. G. Kov\'acs  and C. E Praeger,  Finite permutation groups with large abelian quotients, Pacific J. Math. 136 (1989), no. 2, 283-–292. 
\bibitem{pak} A. Lubotzky,
The expected number of random elements to generate a finite group, J. Algebra 257, (2002) 452-495.
\bibitem{sub}
A. Lubotzky and D. Segal, Subgroup growth,
Progress in Mathematics, 212.
Birkh\"auser Verlag, Basel (2003)
\bibitem{al} A. Lucchini,  A bound on the number of generators of a finite group, Arch. Math.  53 (1989), no. 4, 313-317.
\bibitem{mon} A. Lucchini, The expected number of random elements to generate a finite group, Monatsh. Math. DOI 10.1007/s00605-015-0789-5
\bibitem{pom} C. Pomerance, 
The expected number of random elements to generate a finite abelian group,
Period. Math. Hungar. 43, 191-19 (2001)
\bibitem{st} U. Stammbach,
Cohomological characterisations of finite solvable and nilpotent groups,
J. Pure Appl. Algebra 11 (1977/78), no. 1--3, 293--301.
\end{thebibliography}
\end{document}